\documentclass[final,onefignum,onetabnum]{siamart190516}

\usepackage{braket,amsfonts}

\usepackage{array}

\usepackage[caption=false]{subfig}
\usepackage{pgfplots}

\newsiamthm{claim}{Claim}
\newsiamremark{remark}{Remark}
\newsiamremark{hypothesis}{Hypothesis}
\crefname{hypothesis}{Hypothesis}{Hypotheses}

\usepackage{algorithmic}

\usepackage{graphicx,epstopdf}

\Crefname{ALC@unique}{Line}{Lines}

\usepackage{amsopn}

\usepackage{graphicx}

\usepackage{amssymb, pifont, color, graphics, amsmath, amssymb, mathrsfs, amsfonts, multicol, amsmath, tikz, pgfplots, graphicx, lmodern} 
\usepackage{tikz}
\usepackage{hyperref,stackengine}
\hypersetup{
    colorlinks=true,
    linkcolor=blue,
    filecolor=red,      
    urlcolor=red,
}

\usepackage{mathtools} 
\usepackage{esint}

\makeatletter
\newcommand{\leqnomode}{\tagsleft@true\let\veqno\@@leqno}
\newcommand{\reqnomode}{\tagsleft@false\let\veqno\@@eqno}
\makeatother
\usepackage[utf8]{inputenc}

\DeclareMathOperator*{\du}{d\!}

\DeclareMathOperator{\bu}{{\boldsymbol{u}}}
\DeclareMathOperator{\blf}{\boldsymbol{f}}
\DeclareMathOperator{\bw}{\boldsymbol{w}}
\DeclareMathOperator{\bv}{\boldsymbol{v}}

\DeclareMathOperator{\bn}{\boldsymbol{n}}

\DeclareMathOperator{\bphi}{\boldsymbol{\varphi}}

\DeclareMathOperator{\dive}{\mathrm{div}}

\DeclareMathOperator{\ws}{\mathrel{\ensurestackMath{\stackon[1pt]{\rightharpoonup}{\scriptstyle\ast}}}}

\makeatletter

\newlength\normalparindent
{\begin{list}{}%
    {%
      \setlength{\itemindent}{-10pt}%
      \setlength{\leftmargin}{20pt}%
      \setlength{\labelwidth}{.3\normalparindent}%
      \addtolength{\topsep}{-0.5\parskip}%
      \listparindent \normalparindent
      \setlength{\parsep}{\parskip}}}%
  {\end{list}}
\makeatother

\usepackage{xspace}
\usepackage{bold-extra}
\usepackage[most]{tcolorbox}

\colorlet{texcscolor}{blue!50!black}
\colorlet{texemcolor}{red!70!black}
\colorlet{texpreamble}{red!70!black}
\colorlet{codebackground}{black!25!white!25}

\lstdefinestyle{siamlatex}{%
  style=tcblatex,
  texcsstyle=*\color{texcscolor},
  texcsstyle=[2]\color{texemcolor},
  keywordstyle=[2]\color{texemcolor},
  moretexcs={cref,Cref,maketitle,mathcal,text,headers,email,url},
}

\tcbset{%
  colframe=black!75!white!75,
  coltitle=white,
  colback=codebackground, 
  colbacklower=white, 
  fonttitle=\bfseries,
  arc=0pt,outer arc=0pt,
  top=1pt,bottom=1pt,left=1mm,right=1mm,middle=1mm,boxsep=1mm,
  leftrule=0.3mm,rightrule=0.3mm,toprule=0.3mm,bottomrule=0.3mm,
  listing options={style=siamlatex}
}

\newtcblisting[use counter=example]{example}[2][]{%
  title={Example~\thetcbcounter: #2},#1}

\newtcbinputlisting[use counter=example]{\examplefile}[3][]{%
  title={Example~\thetcbcounter: #2},listing file={#3},#1}

\DeclareTotalTCBox{\code}{ v O{} }
{ 
  fontupper=\ttfamily\color{black},
  nobeforeafter,
  tcbox raise base,
  colback=codebackground,colframe=white,
  top=0pt,bottom=0pt,left=0mm,right=0mm,
  leftrule=0pt,rightrule=0pt,toprule=0mm,bottomrule=0mm,
  boxsep=0.5mm,
  #2}{#1}

\patchcmd\newpage{\vfil}{}{}{}
\begin{tcbverbatimwrite}{tmp_\jobname_header.tex}
\title{A Convective Boundary Condition for the Navier-Stokes Equations: Existence Analysis and Numerical Implementations\thanks{
\funding{This work is supported by JSPS KAKENHI Grant Numbers JP18H01135, JP20H01823, JP20KK0058 and JP21H04431, and JST CREST Grant Number JPMJCR2014 for {\bf HN}; and by the Japanese Government (MEXT) Scholarship for {\bf JSS}.}}}

\author{John Sebastian H. Simon\thanks{Division of Mathematical and Physical Sciences, 
              Graduate School of Natural Science and Technology, 
              Kanazawa University, Kanazawa 920-1192, Japan  (\email{john.simon@stu.kanazawa-u.ac.jp},\email{jhsimon1729@gmail.com}).}
\and Hirofumi Notsu\thanks{Faculty of Mathematics and Physics, Kanazawa University, Kanazawa 920-1192, Japan (\email{notsu@se.kanazawa-u.ac.jp}).}}

\headers{A Convective Boundary Condition for the Navier-Stokes Equations}{John Sebastian H. Simon and Hirofumi Notsu}
\end{tcbverbatimwrite}
\input{tmp_\jobname_header.tex}

\ifpdf
\hypersetup{ pdftitle={Navier-Stokes} }
\fi

\begin{document}
\maketitle

\begin{tcbverbatimwrite}{tmp_\jobname_abstract.tex}
\begin{abstract}
 Due to computational complexity, fluid flow problems are mostly defined on a bounded domain. Hence, capturing fluid outflow calls for imposing an appropriate condition on the boundary where the said outflow is prescribed. Usually, the Neumann-type boundary condition called do-nothing condition is the go-to description for such outflow phenomenon However, such condition does not ensure an energy estimate for the Navier--Stokes equations - let alone establish the existence of solutions. In this paper, we analyze a convective boundary condition that will capture outflow and establish the existence of solutions to the governing equation. We shall show existence and uniqueness results for systems with mixed boundary conditions - Dirichlet condition and the convective boundary condition. The first system is a stationary equation where the Dirichlet condition is purely homogeneous, the other is where an input function is prescribed, and lastly a dynamic system with prescribed input function. We end by showing numerical examples to illustrate the difference between the current outflow condition and the usual do-nothing condition.
\end{abstract}

\begin{keywords}
  Navier-Stokes equations, artificial boundary condition
\end{keywords}

\begin{AMS}
  76D05, 35M12, 49K20
\end{AMS}
\end{tcbverbatimwrite}
\input{tmp_\jobname_abstract.tex}

\section{Introduction}

Most simulations of fluid flow call for imposing inflow and outflow conditions.  A good example is in modeling blood flow in arteries, or in modeling the generation of Karman vortex in a fluid flowing through a channel with an obstacle. Nevertheless, both scenarios require an input boundary and an outflow boundary upon which appropriate boundary conditions are supposed to be considered. For the input condition, a non-homogeneous Dirichlet condition may be imposed. The outflow on the other hand, is where the crux of the matter lies. Usually, the outflow profile is described by using the so-called {\it do-nothing} condition which is written by letting the product of the stress tensor and the outward unit normal vector on the outflow boundary be equal to zero. However, due to the nonlinear nature of the Navier-Stokes equations the aforementioned outflow condition is insufficient to ensure a good energy estimate for the solution -- as well as prove its existence.

To bypass this issue, several authors proposed using artificial boundary conditions that will ensure that the governing system is well-posed at the same time capture an outflow behavior on the boundary.  Boyer F.  and Fabrie, P. \cite{boyer2007}, and Brunueau, C.-H. and Fabrie, P. \cite{bruneau1996} proposed several boundary conditions that take into account the nonlinearity of the Navier-Stokes equations. These boundary conditions, however, need to be formulated carefully so as to ensure coercivity of the left hand side of the weak formulation of the system. Among the conditions proposed is of the following form 
	\begin{align*}
		\sigma({\bu},p){\bn} = \frac{1}{2}({\bu}\cdot{\bn})_{-}{\bu},
	\end{align*}
where ${\bu}$ and $p$ are the fluid velocity and pressure, respectively, $\sigma({\bu},p) = 2\nu D(\bu) - pI$ corresponds to the fluid stress tensor, $D({\bu}) = (\nabla{\bu} + \nabla{\bu}^\top)/2$ denotes the deformation rate tensor, $\nu>0$ corresponds to fluid viscosity, ${\bn}$ is the outward unit normal vector on the outflow boundary, and $({\bu}\cdot{\bn})_{-}$ is the negative part of the product ${\bu}\cdot{\bn}$, i.e., 
\begin{align*}
	({\bu}\cdot{\bn})_{-} = \left\{ \begin{aligned} &0&&\text{    when }{\bu}\cdot{\bn}\ge 0,\\
	&{\bu}\cdot{\bn}&&\text{    otherwise}. \end{aligned} \right.
\end{align*}
This condition is also called the directional do-nothing condition by Braack, M. and Mucha, P.  in \cite{braack2014} where they considered homogeneous Dirichlet conditions on the boundaries of the domain excluding the outflow boundary. 

As pointed out in \cite{braack2014} the directional do-nothing reflects the same outflow profile as the usual do-nothing condition for Poiseuille flows, but poses an enhanced stability as compared to the usual method. However, the directional do-nothing condition presents challenges when utilized as governing states to fluid control problems especially to the corresponding adjoint system of the optimization problem.  Nevertheless, such condition have been applied to several physical phenomena, such as ferrofluid flows \cite{ELSHEHABEY2020}, two-phase flows with a saturated version of the discontinuous nature of $(\cdot)_{-}$ \cite{DONG2014}, and a Signorini type unilateral outflow \cite{zhou2016}, to name a few. Such condition is also considered as a backflow stabilization technique, which is important for applications such as hemodynamics \cite{ARBIA2016,Formaggia2007}.

Another boundary condition that has been considered for outflow behavior is by taking into account the total pressure, which was first introduced in \cite{begue1987}. The formulation of this boundary condition is carried out in a more {\it natural} way than the directional do-nothing condition, by which we capitalize on the fact that 
\[
({\bu}\cdot\nabla){\bu} = (\nabla\times{\bu})\times{\bu} + \frac{1}{2}\nabla|{\bu}|^2.
\]
The boundary condition can then be written as $\left(\sigma({\bu},p)-  \frac{1}{2}|{\bu}|^2I\right){\bn} =0$. This particular boundary condition has been illustrated to feature a good outflow behaviour \cite{heywood1992}. The existence and regularity of solutions to the Navier--Stokes equations governed with such condition has been studied in \cite{nowakowski2020}. Due to its physical relevance on outflow phenomenon, the said condition has been applied for example to hemodynamics, see \cite{ARBIA2016,Formaggia2007,Incaux2014} among others, and to some homogenization problems, among them is \cite{Carbou2008}.


In this paper, we shall discuss another natural boundary condition that behaves in the same way as that of the total pressure approach but can be interpreted as boundary convection.  In particular, we shall consider the following outflow condition
\begin{align*}
	\sigma({\bu},p){\bn} = \frac{1}{2}({\bu}\otimes{\bu}){\bn},
\end{align*}
where the tensor product $\otimes$ is defined as ${\bu}\otimes{\bv} = [u_iv_j]_{i,j=1}^2$, with ${\bu} = (u_1,u_2)$, and ${\bv} = (v_1,v_2)$. We note that the condition of interest -- which we shall call a {\it convective boundary condition} (CBC) -- comes out mathematically natural in the convection term as shown in the computations of Brunueau, C.-H. and Fabrie, P. \cite{bruneau1996}. Furthermore, this condition ensures one of the solvability of the linearized and adjoint system if one wishes to consider an outflow condition for a fluid control problem. We also mention that a similar Robin type boundary condition has been mentioned in \cite{gresho2000} but no analysis has been presented for such condition.

This paper is organized as follows: in the next section, we shall discuss some preliminary concepts such as functional spaces and some notations that will be used. Section \ref{sec3} will be dedicated to the analysis of the stationary problem, this includes homogeneous Dirichlet boundary condition on the boundaries of the domain except the outflow boundary and a non-homogeneous Dirichlet data on a specified non-outflow boundary. The analysis of the time dependent version of the problem discussed in Section 3 will by done on Section \ref{sec4}. We provide some numerical examples on Section \ref{sec5} and we finish the paper by providing some concluding remarks on the last section.


\section{Preliminaries}\label{sec2}

For a Banach space $X$, its dual is denoted as $X^*$ and their pairing is denoted as $\langle x^*,x\rangle_{X^*\times X}$ for $x^*\in X^*$ and $x\in X$. For a domain $\Omega\subset \mathbb{R}^2$,  and $d = 1,2$ we denote by $L^p(\Omega;\mathbb{R}^d)$ for $p\ge 1$ the space of p-Lebesgue functions, and we shall also use the standard notation for Sobolev spaces as $W^{s,p}(\Omega;\mathbb{R}^d)$ for $p\ge 1$ and $k\ge 0$, with $W^{s,2}(\Omega) = H^s(\Omega;\mathbb{R}^d).$ 
Let $\Gamma_0$ be a portion of the boundary with non-zero measure, we consider the following spaces for Neumann/Robin-type boundary conditions:
\begin{align*}
	\mathcal{W}(\Omega;\mathbb{R}^d) & : = \{\bphi\in C^\infty(\Omega;\mathbb{R}^d): \varphi=0\text{ on a neighborhood of }\Gamma_0 \},\\
	H_{\Gamma_0}^s(\Omega;\mathbb{R}^d) &:=  \overline{\mathcal{W}(\Omega;\mathbb{R}^d)}^{\|\cdot\|_{H^s(\Omega;\mathbb{R}^d)}}.
\end{align*}

To take into account the divergence-free property of the fluid velocity we shall utilize the following solenoidal spaces:
\begin{align*}
	W & : = \{ {\bphi}\in \mathcal{W}(\Omega;\mathbb{R}^2): \nabla\cdot{\bphi}=0\text{ in }\Omega \},\\
	V & := \{ {\bphi}\in H_{\Gamma_0}^1(\Omega;\mathbb{R}^2): \nabla\cdot{\bphi}=0\text{ in }\Omega \},\\
	H & := \{ {\bphi}\in L^2(\Omega;\mathbb{R}^2): \nabla\cdot{\bphi} = 0\text{ in }\Omega, {\bphi}\cdot{\bn} = 0\text{ on }\Gamma_0 \}.
\end{align*}
The spaces $V$ and $H$ satisfy the Gelfand triple property, i.e., $V\hookrightarrow H\hookrightarrow V^*$, with the first embedding known to be dense and continuous. Furthermore, $W$ is dense in $V$. 

For an interval $I\subset \mathbb{R}$ and a real Banach space $X$, we shall consider the space of continuous functions from $I$ to $X$ denoted by $C(I;X)$ with its usual norm $\sup_{t\in I}\|u(t)\|_X$.  We shall also consider the Bochner spaces $L^p(I;X)$ for $p\ge 1$ with the norms
	\begin{align*}
		\|u\|_{L^p(I;X)} = \left\{ \begin{aligned} 
		&\mathrm{ess}\sup_{t\in I}\|u(t)\|_X &&\text{for }p=\infty,\\
		&\left(\int_I \|u(t)\|_X^p \du t \right)^{\!1/p} &&\text{otherwise}.
		\end{aligned} \right.
	\end{align*}
Lastly, thanks to Aubin-Lions Lemma \cite{simon1986} , the space $W^p_I(V):= \{{\bu} \in L^2(I,V); \partial_t{\bu}\in L^p(I;{V}^*) \}$ is compactly embedded to $L^p(I;H)$, and we have the following inclusion $W^2_I(V) \subset C(\overline{I};H)$. In this space, we shall also consider the norm 
\begin{align*}
	\|{\bu}\|_{W^p_I(V)} = \|\bu \|_{L^2(I,V)} + \|\partial_t{\bu}\|_{L^p(I;{V}^*)}.
\end{align*}
In the subsequent sections we shall use the notation $W^p(V)$ if the interval $I=(0,T)$ is used.

The following operators will also be useful to simplify the analyses that will be done in the subsequent sections. The operator $a_0:H^1(\Omega;\mathbb{R}^2)\times H^1(\Omega;\mathbb{R}^2)\to \mathbb{R}$ is defined as 
\begin{align*}
	a_0({\bu},{\bv}) = 2\int_\Omega D({\bu}):D({\bv})\du x.
\end{align*}
Meanwhile, we shall also consider the trilinear form $a_1:H^1(\Omega;\mathbb{R}^2)\times H^1(\Omega;\mathbb{R}^2)\times H^1(\Omega;\mathbb{R}^2)\to \mathbb{R}$ given by
\begin{align*}
	a_1({\bw};{\bu},{\bv}) = \int_\Omega [({\bw}\cdot\nabla){\bu}]\cdot{\bv}\du x - \frac{1}{2}\int_{\partial\Omega\backslash\Gamma_0} ({\bw}\cdot{\bn})({\bu}\cdot{\bv}) \du s.
\end{align*}

The following properties of $a_1$ are crucial for showing existence of solutions to the systems which we shall be studying.  
\begin{lemma}
Let ${\bu},{\bv},{\bw}\in V$, then the following identities hold true:
\begin{enumerate}
	\item[(i)] $\displaystyle a_1({\bw};{\bu},{\bv}) + a_1({\bw};{\bv},{\bu}) =0$;
	\item[(ii)] $\displaystyle a_1({\bv};{\bu},{\bu}) = 0$.
\end{enumerate}
Furthermore, the first part of the trilinear form $a_1(\cdot;\cdot,\cdot)$ satisfies the inequality
\begin{align}
\left| \int_\Omega [({\bw}\cdot\nabla){\bu}]\cdot{\bv}\du x \right| \le c\|{\bw}\|_H^{1/2}\|{\bw}\|_V^{1/2}\|{\bu}\|_V\|{\bv}\|_H^{1/2}\|{\bv}\|_V^{1/2}.
\label{estimate:trilinear}
\end{align}
\label{lemma:trilinearpropoerties}
\end{lemma}

\section{Stationary Problem}\label{sec3}

In this section we shall analyze the stationary Navier-Stokes equations upon which the convective boundary condition is imposed. Although we shall consider mixed boundary conditions, this section is divided into two: first, we consider a system where a homogeneous Dirichlet and the convective boundary conditions are imposed on two parts of the boundary of the domain; and the other is where we consider a Dirichlet data on a portion of the non-outflow boundary.

\subsection{Homogeneous Dirichlet}
\label{section:2.1}
Let $\Omega\subset\mathbb{R}^2$ be a bounded domain with boundary $\partial\Omega$. Let us also consider partition of the boundary denoted by $\Gamma_0$ and $\Gamma_1$, both of which have non-zero measures. We shall show the existence of a velocity field $\bu:\Omega\to\mathbb{R}^2$ and pressure $p:\Omega\to\mathbb{R}$ that satisfy the following Navier-Stokes equations
\begin{align}
	\left\{
		\begin{aligned}
			-\nabla\cdot\sigma({\bu},p) + ({\bu}\cdot\nabla){\bu} & = {\blf} &&\text{in }\Omega,\\
			\nabla\cdot{\bu}& = 0 &&\text{in }\Omega,\\
			{\bu} & = 0 &&\text{on }\Gamma_0,\\
			\sigma({\bu},p){\bn} & = \frac{1}{2}({\bu}\otimes{\bu}){\bn} &&\text{on }\Gamma_1.
		\end{aligned}
	\right.
	\label{system:hdirichlet}
\end{align}

The variational form of \eqref{system:hdirichlet} can be written -- by virtue of Green's identities -- as follows: {\it Find } ${\bu}\in V$ {\it that solves the equation }
\begin{align}
	\nu a_0({\bu},{\bv}) + a_1({\bu};{\bu},{\bv})   = \langle{\blf},{\bv}\rangle_{V^*\times V}\quad \forall {\bv}\in V.
	\label{weak:hdirichlet}
\end{align}

Any function ${\bu}\in V$ that satisfies \eqref{weak:hdirichlet} is called a {\it weak solution} of the Navier-Stokes equation \eqref{system:hdirichlet}.  To establish the existence of the weak solution we shall utilize Theorem 1.2 in \cite[Section IV p. 280]{girault1986}. Meaning to say, we shall show that the left hand side of \eqref{weak:hdirichlet} is coercive, and that the map 
\[
{\bu}\mapsto \nu a_0({\bu},{\bv}) + a_1({\bu};{\bu},{\bv}) 
\]
 is sequentially weakly continuous for all ${\bv}\in V$, that is, if $\{{\bu}_n\}\subset V$ is a sequence that converges weakly to ${\bu}\in V$, then 
\[
\lim_{n\to\infty}\nu a_0({\bu}_n,{\bv}) + a_1({\bu}_n;{\bu}_n,{\bv}) = \nu a_0({\bu},{\bv}) + a_1({\bu};{\bu},{\bv}) .
\]

\begin{theorem}
 Let $\Omega\subset \mathbb{R}^2$ be of class $\mathcal{C}^1$, and suppose that the source function satisfies ${\blf}\in V^*$. Then, there exists an element ${\bu}\in V$ that satisfies \eqref{weak:hdirichlet}, such that 
 \begin{align}
 \|{\bu}\|_{V} \le c\|{\blf}\|_{V^*}
 \label{energy:hdirichlet}
 \end{align}
  for some constant $c>0$.
 \label{theorem:wphdirichlet}
\end{theorem}

\begin{proof}
	We start by showing that the left hand side of \eqref{weak:hdirichlet} is coercive. We do this by utilizing Lemma \ref{lemma:trilinearpropoerties} (ii). Indeed, for any ${\bu}\in V$, we have
	\begin{align*}
		\nu a_0({\bu},{\bu}) + a_1({\bu};{\bu},{\bu}) = 2\nu \int_\Omega D({\bu}):D({\bu})\du x = \nu \|{\bu}\|^2_V.
	\end{align*}

	As for the second property, let ${\bv}\in V$ and $\{{\bu}_n \}\subset V$ be a sequence weakly converging to ${\bu}\in V$.  One can easily show that $a_0({\bu}_n,{\bv})\to a_0({\bu},{\bv})$, so what remains for us to show is that 
	\begin{align}
		a_1({\bu}_n;{\bu}_n,{\bv})\to a_1({\bu};{\bu},{\bv}).
		\label{convergence:trilinear}
	\end{align}
Let us first point out that due to Rellich-Kondrachov embeding theorem (see, e.g. \cite[Part I, Theorem 6.3]{adams2003}),  we get the following convergences 
	\begin{align}
		\left\{
			\begin{aligned}
				& {\bu}_n\to {\bu} &&\text{in }H,\\
				& {\bu}_n\to {\bu} &&\text{in }L^q(\Gamma_1;\mathbb{R}^2) \ (q\ge 2).
			\end{aligned}
		\right.
		\label{convergence:rellich}
	\end{align}
Using the definition of the operator $a_1(\cdot;\cdot,\cdot)$, and Lemma \ref{lemma:trilinearpropoerties} we get the following estimate.
	\begin{align*}
		|a_1({\bu}_n;{\bu}_n,{\bv}) -  a_1({\bu};{\bu},{\bv})| = &\,  |a_1({\bu}_n;{\bv},{\bu}_n)- a_1({\bu};{\bv},{\bu})|\\
		 \le&   |a_1({\bu}_n-{\bu};{\bv},{\bu}_n)| +  |a_1({\bu};{\bv},{\bu}_n - {\bu})|\\
		= & \left| \int_\Omega [(({\bu}_n-{\bu})\cdot\nabla){\bv}]\cdot{\bu}_n\du x - \frac{1}{2}\int_{\Gamma_1} (({\bu}_n-{\bu})\cdot{\bn})({\bv}\cdot{\bu}_n) \du s\right| \\
		& \, + \left| \int_\Omega [({\bu}\cdot\nabla){\bv}]\cdot({\bu}_n-{\bu})\du x - \frac{1}{2}\int_{\Gamma_1} ({\bu}\cdot{\bn})({\bv}\cdot({\bu}_n-{\bu})) \du s\right|\\
		\le &\, c\| {\bu}_n-{\bu}\|_H^{1/2}(\|{\bu_n}\|_V + \|{\bu}\|_V)^{1/2}\|{\bv}\|_V\|{\bu}_n\|_V\\
		&\, + c\|{\bu}\|_V\|{\bv}\|_V\| {\bu}_n-{\bu}\|_H^{1/2}(\|{\bu_n}\|_V + \|{\bu}\|_V)^{1/2}\\
		&\, + \frac{c}{2}\|{\bu}_n-{\bu} \|_{L^2(\Gamma_1;\mathbb{R}^2)}\|{\bv}\|_{L^4(\Gamma_1;\mathbb{R}^2)}(\|{\bu}_n \|_{L^4(\Gamma_1;\mathbb{R}^2)}+\|{\bu} \|_{L^4(\Gamma_1;\mathbb{R}^2)}).
	\end{align*}

From \eqref{convergence:rellich}, the right hand side of the previous computation converges to zero, which proves the sequential weak continuity asked, and hence the existence of an element ${\bu}\in V$ that solves \eqref{weak:hdirichlet}.
	
	Lastly, by letting ${\bv} = {\bu} \in V$ in \eqref{weak:hdirichlet}, we get
	\begin{align*}
		\nu \|{\bu}\|_V^2 = \langle{\blf},{\bu}\rangle_{V^*\times V}\le c_1\|{\blf}\|_{V^*}\|{\bu}\|_V,
	\end{align*}
	where $c_1>0$ is the Poincar{\'e} constant.
	Therefore, we get the energy estimate mentioned in Theorem \ref{system:hdirichlet} with $c=c_1/\nu$.
\end{proof}

For the uniqueness of the weak solution, let us introduce the following notation which is the norm of the trilinear form $a_1$.
\begin{align}
	\mathcal{B} = \sup_{{\bu},{\bv},{\bw}\in V}\frac{a_1({\bu};{\bv},{\bw})}{\|{\bu}\|_V\|{\bv}\|_V\|{\bw}\|_V}.
	\label{norm:trilinear}
\end{align}
The existence of such quantity can be established due to the continuity of the trilinear form in $V\times V\times V$ (see equation \eqref{estimate:trilinear} in Lemma \ref{lemma:trilinearpropoerties}).

\begin{theorem}
	Suppose that the assumptions in Theorem \ref{theorem:wphdirichlet} hold, and that the the following estimate holds:
	\begin{align}
		\mathcal{B}\|{\blf}\|_{V^*}c_1 < \nu^2,
		\label{estimate:uniquenesshdirichlet}
	\end{align}
	where $c_1>0$ is the same constant as in the proof of Theorem \ref{theorem:wphdirichlet}.
	Then the solution ${\bu}\in V$ to \eqref{weak:hdirichlet} is unique.
	\label{theorem:uniquenesshdirichlet}
\end{theorem}

\begin{proof}
	Suppose that there exists two solutions $ {\bu}_1,{\bu}_2\in V $ for \eqref{weak:hdirichlet}, this implies that $\nu \|{\bu}_i\|_{V} \le c_1\|{\blf}\|_{V^*}$ for $i=1,2$ and that the element ${\bw} = {\bu}_1 - {\bu}_2 \in V$ satisfies 
	\begin{align}
		\nu a_0({\bw},{\bv}) + a_1({\bu}_1;{\bw},{\bv}) = a_1({\bw};{\bv},{\bu}_2)\quad \forall{\bv}\in V.
		\label{weak:differencehdirichlet}
	\end{align}
	
By taking ${\bv} = {\bw}$ in \eqref{weak:differencehdirichlet}, and utilizing Lemma \ref{lemma:trilinearpropoerties}({\it ii}) we get the following equation
\begin{align} 
	\nu\|{\bw}\|_V^2 = a_1({\bw};{\bw},{\bu}_2).
	\label{equation:9}
\end{align}
From \eqref{norm:trilinear}, and due to the energy estimate of weak solution, we further infer that
\begin{align*}
	\nu\|{\bw}\|_V^2 \le \mathcal{B}\|{\bw}\|_V^2 \|{\bu}_2\|_V \le \frac{\mathcal{B}\|{\blf}\|_{V^*}c_1}{\nu}\|{\bw}\|_V^2.
\end{align*}
This implies that $( \nu^2-\mathcal{B}\|{\blf}\|_{V^*}c_1)\|{\bw}\|_V^2 \le 0$, and from the assumption \eqref{estimate:uniquenesshdirichlet} we conclude that ${\bw} = 0$.
\end{proof}

\subsection{System with non-zero Dirichlet data}
\label{section:2.2}
In this section, we shall discuss perhaps what evaded most to consider the boundary condition we are interested in, i.e., a system coupled with a non-homogeneous Dirichlet condition. In particular, we shall consider the following system
\begin{align}
	\left\{
	\begin{aligned}
		-\nabla\cdot\sigma({\bu},p) + ({\bu}\cdot\nabla){\bu} & = {\blf} &&\text{in }\Omega,\\
		\nabla\cdot{\bu} & = 0 &&\text{in }\Omega,\\
		{\bu} & = 0 &&\text{on }\Gamma_H,\\
		{\bu} & = {\bu}_{in} &&\text{on }\Gamma_N,\\
		\sigma({\bu},p){\bn} & =\frac{1}{2}({\bu}\otimes{\bu}){\bn} &&\text{on }\Gamma_1.
	\end{aligned}
	\right.
	\label{system:nhdirichlet}
\end{align}

For the purpose of uniformity, we shall denote by $\Gamma_0 = \overline{\Gamma_H}\cup \overline{\Gamma_N}$ the boundaries of $\Omega$ except that of the outflow boundary, where the subscripts ${}_H$ and ${}_N$ stand for the homogeneous and non-homogeneous boundaries, respectively. Furthermore, since one of the main point of this system is to model outflow/inflow phenomena, we shall consider $\Gamma_H$ to constitute the walls of a channel, and if a body inside a channel is also considered the boundary of such body is also included in the definition of $\Gamma_H$ (see Figure \ref{fig1} for illustrations of such boundaries).

\begin{figure}[h!]
 \centering
  \includegraphics[width=.45\textwidth]{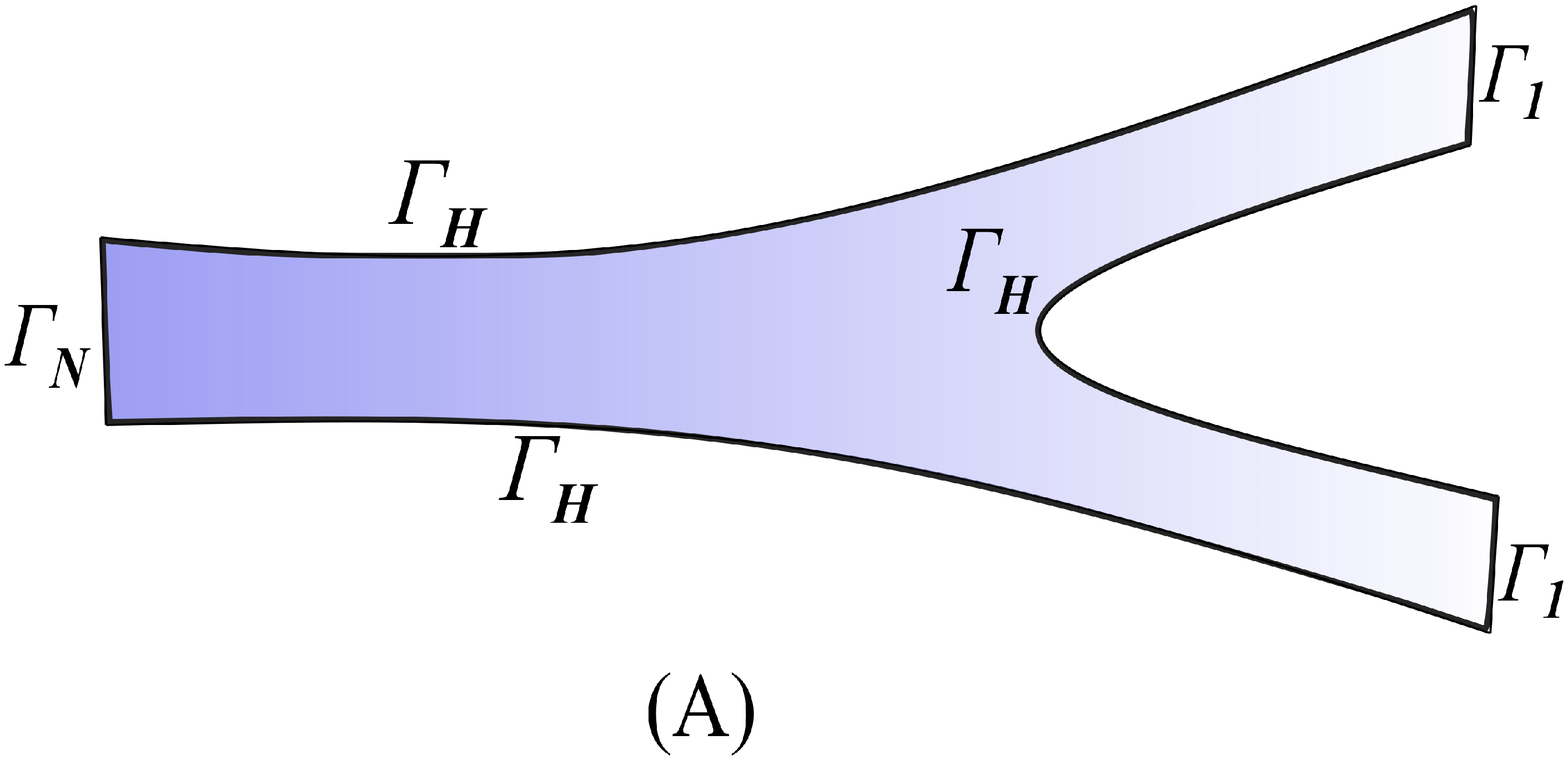} \includegraphics[width=.45\textwidth]{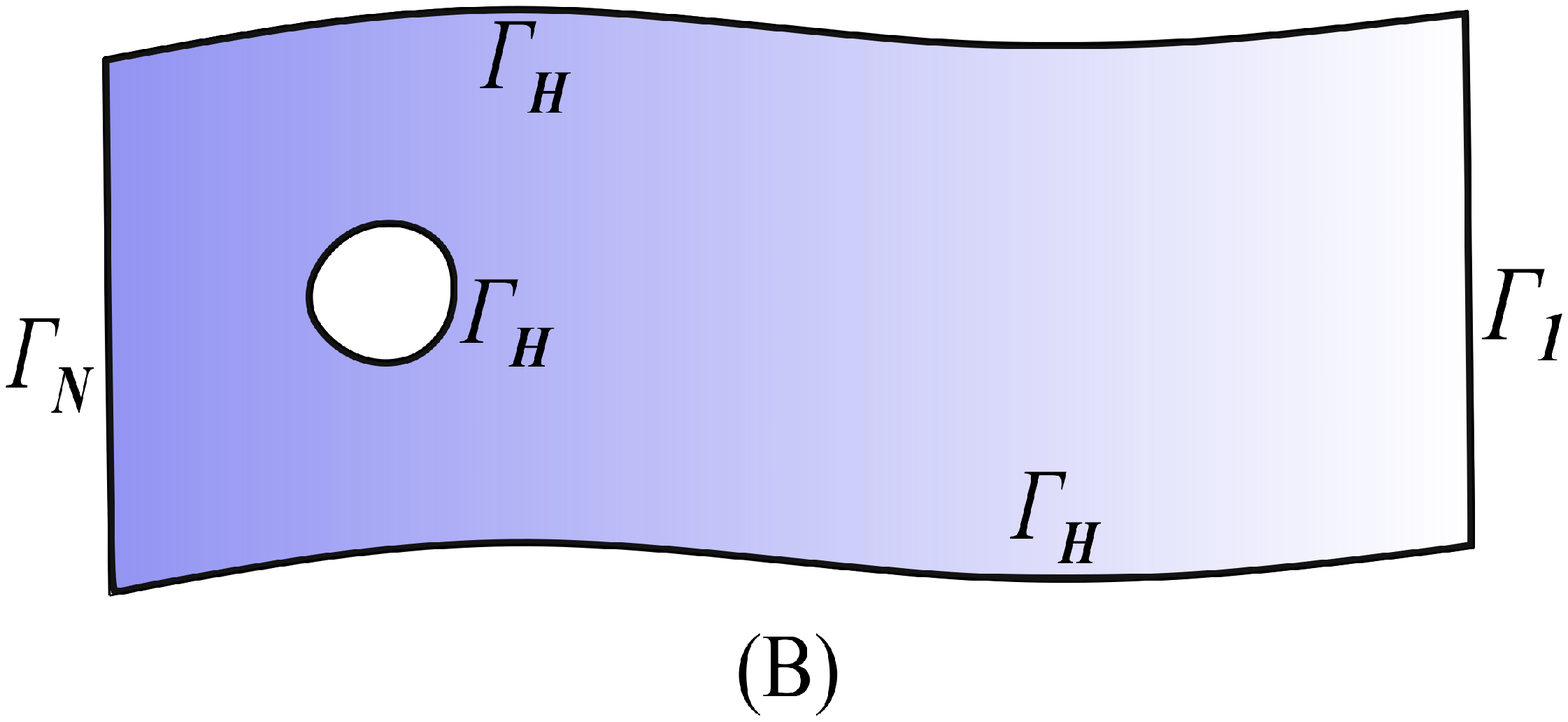}
 \caption{Illustrations of the boundaries/domains where (A) purely inflow/outflow phenomenon, and (B) flow past an obstacle are considered.}
\label{fig1}
 \end{figure}
 
The challenge in establishing the existence of the weak solution to \eqref{system:nhdirichlet} is on proving the coercivity of the left hand side of its variational form. This obstacle will be circumvented by the following lemma. Furthermore, this lemma helps in lifting the input function ${\bu}_{in}$ over the whole domain $\Omega$. Before we introduce the said result, let us mention that the input function will be assumed to satisfy ${\bu}_{in}\in H^{1/2}(\Gamma_N)$ and 
\begin{align}
\int_{\Gamma_N}{\bu}_{in}\cdot {\bn}\du s = 0.
\label{equality:nhomogeneous0}
\end{align}
\begin{lemma}
	Let $\epsilon > 0$, then there exists ${\bw}_0 := {\bw}_0(\epsilon) \in H^1(\Omega;\mathbb{R}^2)$ such that
	\begin{align}
	\left\{
		\begin{aligned}
			\nabla\cdot{\bw}_0 & = 0 &&\text{in }\Omega,\\
			{\bw}_0 & = 0 &&\text{on }\Gamma_H,\\
			{\bw}_0 & = {\bu}_{in} &&\text{on }\Gamma_N.
		\end{aligned}
	\right.
	\label{properties:midinp}
	\end{align}
	Furthermore, we get the following nonlinear estimate
	\begin{align}
		|a_1({\bv};{\bw}_0,{\bv})| \le \epsilon\|{\bv}\|_V^2\quad \forall{\bv}\in V.
		\label{estimate:midinp}
	\end{align}
	\label{lemma:midinp}
\end{lemma}
 
We use the same methods as in the proof of \cite[Lemma IV.2.3]{girault1986} but had to make sure that the estimate \eqref{estimate:midinp} holds as well for the boundary integral.
 
 \begin{proof}
	
	From \eqref{equality:nhomogeneous0} there exists ${\bu}_0 \in H^1(\Omega;\mathbb{R}^2)$ such that $\dive{\bu}_0 = 0$ in $\Omega$, and ${\bu}_0|_{\Gamma_N} = {\bu}_{in}$ and ${\bu}_0|_{\partial\Omega\backslash\Gamma_N} = 0$. From \cite[Theorem I.3.1]{girault1986}, ${\bu}_0 = \nabla\times\phi$ for some stream function $\phi\in H^2(\Omega;\mathbb{R})$.  Furthermore,  we can choose a particular stream function such that $\phi|_{\Gamma_H} = 0$. Let us then define ${\bw}_{0,\delta}\in H^1(\Omega;\mathbb{R}^2)$ as ${\bw}_{0,\delta} = \nabla\times(\theta_\delta \phi)$, where $\theta_\delta\in C^2(\overline{\Omega};\mathbb{R})$ is the function that satisfies ( see \cite[Lemma IV.2.4]{girault1986})
	\begin{align*}
		\left\{
			\begin{aligned}
				&\theta_\delta = 1 &&\text{in a neighborhood of }\Gamma_0,\\
				&\theta_\delta(x) = 0 &&\text{for }d(x,\Gamma_0) \ge 2e^{-1/\delta},\\
				&|\partial\theta_\delta/\partial x_i| \le \delta/d(x,\Gamma_0) &&\text{for }d(x,\Gamma_0) \le 2e^{-1/\delta}.
			\end{aligned}
		\right.
	\end{align*}
	
	It can be easily shown that
	\begin{align}
	 \|v_iw_j\|_{L^2(\Omega;\mathbb{R})}\le c_\delta|v_i|_{H^1(\Omega;\mathbb{R})},
	 \label{l2est}
	 \end{align}
	where $v_i$ and $w_j$, for $i,j=1,2$, are such that ${\bv}=(v_1,v_2)\in H_{\Gamma_0}^1(\Omega;\mathbb{R}^2)$ and ${\bw}_{0,\delta}=(w_1,w_2)$, and the constant $c_\delta>0$ is dependent on $\delta>0$ in such a way that $c_\delta \to 0$ as $\delta\to0$.
	Now,  from Lemma \ref{lemma:trilinearpropoerties}(i) we get 
	\begin{align*}
		|a_1({\bv};{\bw}_{0,\delta},{\bv})| &  = |({\bv}\cdot\nabla{\bv},{\bw}_{0,\delta})_{\Omega}- \frac{1}{2}({\bv}\cdot{\bn}, {\bv}\cdot{\bw}_{0,\delta})_{\Gamma_{\rm out}}|\\
		&\le \left| \sum_{i,j=1}^2 \int_{\Omega} v_iw_j\frac{\partial v_j}{\partial x_i} \du x\right| +  \frac{1}{2}\left| \int_{\Gamma_{1}} ({\bv}\cdot{\bn})({\bv}\cdot{\bw}_{0,\delta}) \du s \right|.
	\end{align*}
	
	Using \eqref{l2est}, the first expression on the last line of the previous computation can be estimated as follows:
	\begin{align}
		\left| \sum_{i,j=1}^2 \int_{\Omega} v_iw_j\frac{\partial v_j}{\partial x_i} \du x\right| & \le c_{1,\delta}\|\bv\|_{V}^2,\label{est:trili}
	\end{align}
	where $ c_{1,\delta} > 0$ is such that $c_{1,\delta} \to 0$ as $\delta\to0$.
	
	As for the boundary integral,  we use the following Hardy inequality \cite[Theorem 330]{hardy1934}:
	\begin{lemma}
		Let $p> 1$, and $\delta \in (0,\infty]$. For any $u\in W^{1,p}(0,\delta)$ such that $u(\delta) = 0$, the following inequality holds:
		\begin{align*}
			\int_0^\delta \frac{|u|^p}{t^p}\du t \le \left| \frac{p}{p-1} \right|^p\int_0^\delta {|u'|^p}\du t.
		\end{align*}
		\label{lemma:hardy}
	\end{lemma}
	
	Now, from the properties of $\theta_\delta$, and by denoting $\Gamma_{1}^{\delta} := \{s\in\Gamma_{1}: d(s,\Gamma_0) \le 2e^{-1/\delta} \}$, we get
	\begin{align}
		\begin{aligned}
		\|{\bw}_{0,\delta}\|_{L^2(\Gamma_{1};\mathbb{R}^2)}  & = \left( \int_{\Gamma_{\rm out}^\delta} |{\bw}_{0,\delta}|^2 \du s \right)^{1/2} \le c \left( \int_{\Gamma_{1}^\delta}  \delta^2\frac{|\phi|^2}{d(s,\Gamma_0)^2} + \left|\nabla\phi\right|^2 \du s \right)^{1/2}\\
			& \le c_1\delta\left( \int_{\Gamma_{1}^\delta}  \frac{|\phi|^2}{d(s,\Gamma_0)^2} \du s \right)^{1/2} + c_2\|\nabla\phi\|_{L^2(\Gamma_{1}^\delta;\mathbb{R}^2)}.
		\end{aligned}	\label{est:w0}
	\end{align}
	
	Due to the regularity assumption on the domain $\Omega$ and since $\phi = 0$ on $\Gamma_{H}$ which is adjacent to the boundary $\Gamma_{1}^{\delta}$,  we can write the boundary integral as the one-dimensional integral
	\[ \int_0^{\Delta(\delta)} \left|\frac{\phi(t)}{t}\right|^2 \du t, \]
	where $\Delta(\delta)$ corresponds to the arc length of the boundary $\Gamma_{1}^{\delta}$. From Lemma \ref{lemma:hardy}, we get 
	\[ \int_0^{\Delta(\delta)} \left|\frac{\phi(t)}{t}\right|^2 \du t \le 4  \int_0^{\Delta(\delta)} \left|\phi'(t)\right|^2 \du t. \]
	This implies that the estimate \eqref{est:w0} can further be estimated as 
	\begin{align}
		\begin{aligned}
		\|{\bw}_{0,\delta}\|_{L^2(\Gamma_{1};\mathbb{R}^2)}  & \le (c_1\delta + c_2)\|\nabla\phi\|_{L^2(\Gamma_{1}^\delta;\mathbb{R}^2)}.
		\end{aligned}	
	\end{align}
	Since $\|\nabla\phi\|_{L^2(\Gamma_{1}^\delta;\mathbb{R}^2)} \to 0$ as $\delta\to 0$,  $ c_{2,\delta} := (c_1\delta + c_2)\|\nabla\phi\|_{L^2(\Gamma_{1}^\delta;\mathbb{R}^2)}\to 0$ as $\delta\to0$.  Thus, we get -- with the help of the Rellich-Kondrachov embedding theorem -- that
	\begin{align}
		\begin{aligned}
		\left| \int_{\Gamma_{1}} ({\bv}\cdot{\bn})({\bv}\cdot{\bw}_{0,\delta}) \du s \right| \le&\, c\|({\bv}\cdot{\bn}){\bv}\|_{L^2(\Gamma_1;\mathbb{R})}\|{\bw}_{0,\delta} \|_{L^2(\Gamma_{1};\mathbb{R})}\\ 
		\le&\,  c\|{\bv}\|_{L^4(\Gamma_{1};\mathbb{R}^2)}^2\|{\bw}_{0,\delta}\|_{L^2(\Gamma_{1};\mathbb{R}^2)}\\
		\le&\,  c_{3,\delta}\|{\bv}\|_{{V}}^2.
		\end{aligned}\label{est:tribdy}
	\end{align}
	where $c_{3,\delta} \to 0$ as $\delta\to 0$.

	From \eqref{est:trili} and \eqref{est:tribdy},  we get 
	\begin{align*}
		|a_1({\bv};\nabla{\bw}_{0,\delta},{\bv})) |  \le \left(c_{1,\delta} + \frac{c_{3,\delta}}{2}\right)\|{\bv}\|_{{V}}^2.
	\end{align*}
	Since $\left(c_{1,\delta} + \frac{c_{3,\delta}}{2}\right)\to 0$ as $\delta\to 0$, we can choose $\delta>0$ small enough so that $\left(c_{1,\delta} + \frac{c_{3,\delta}}{2}\right)\le \epsilon$, and with this choice of $\delta$ we take ${\bw}_{0} = {\bw}_{0,\delta}$. 
	\end{proof}

%
 
 For an arbitrary $\epsilon >0$ and ${\bu}_0\in H^1(\Omega;\mathbb{R}^2)$ from Lemma \ref{lemma:midinp}, we call the element $\tilde{\bu}\in V$ the perturbed weak solution of \eqref{system:nhdirichlet} if it satisfies the equation
 \begin{align}
 	\nu a_0(\tilde{\bu},{\bv}) + a_1(\tilde{\bu};\tilde{\bu},{\bv}) + a_1({\bw}_0;\tilde{\bu},{\bv}) + a_1(\tilde{\bu};{\bw}_0,{\bv}) = \langle \Phi,{\bv}\rangle_{V^*\times V},
 	\label{weak:nhdirichlet}
 \end{align}
where $\Phi\in V^*$ is defined as 
\begin{align*}
	 \langle \Phi,{\bv}\rangle_{V^*\times V} = \langle{\blf},{\bv}\rangle_{V^*\times V} - \nu a_0({\bw}_0,{\bv}) - a_1({\bw}_0;{\bw}_0,{\bv}).
\end{align*}
Note that the element ${\bu} = \tilde{\bu} + {\bw}_0$ can be regarded as the weak solution to the system \eqref{system:nhdirichlet}.

\begin{theorem}
	Let $\Omega\subset\mathbb{R}^2$ be of class $\mathcal{C}^1$, ${\blf}\in V^*$, and ${\bu}_{in}\in H^{1/2}(\Gamma_N)$ satisfies \eqref{equality:nhomogeneous0}. Then, the perturbed weak solution $\tilde{\bu}\in V$ of \eqref{weak:nhdirichlet} exists and satisfies
	\begin{align}
		\|\tilde{\bu}\|_{V} \le c(\|\blf\|_{V^*} + (\nu +\|{\bw}_0\|_{H^1(\Omega;\mathbb{R}^2)}^2)\|{\bw}_0\|_{H^1(\Omega;\mathbb{R}^2)}^2),
		\label{energy:nhdirichlet}
	\end{align}
	for some constant $c>0$.
	\label{theorem:wpnhdirichlet}
\end{theorem}

\begin{proof}
The proof of the theorem - just as we have done in Theorem \ref{theorem:wphdirichlet} - will be divided into two parts: first is to show that the left hand side of \eqref{weak:nhdirichlet} is coercive; the next one is to show that the map 
\begin{align*}
	\tilde{\bu}\mapsto \nu a_0(\tilde{\bu},{\bv}) + a_1(\tilde{\bu};\tilde{\bu},{\bv}) + a_1({\bw}_0;\tilde{\bu},{\bv}) + a_1(\tilde{\bu};{\bw}_0,{\bv})
\end{align*}
is sequentially weakly continuous.

The latter step can be done similarly with that of the previous section, so we shall only delve into proving the coercivity.

Indeed, by taking $\bv = \tilde{\bu}$ on the left hand side of \eqref{weak:nhdirichlet}, and by utilizing \eqref{estimate:midinp} with $\epsilon = \nu/2$ we get
\begin{align*}
	\nu\|\tilde{\bu}\|_V^2 + a_1(\tilde{\bu};{\bw}_0,\tilde{\bu}) \ge \frac{\nu}{2}\|\tilde{\bu}\|_V^2.
\end{align*}
As for the energy estimate \eqref{energy:nhdirichlet}, taking $\bv = \tilde{\bu}$ on the right hand side of \eqref{weak:nhdirichlet} gives us
	\begin{align*}
		| \langle \Phi,\tilde{\bu}\rangle_{V^*\times V}| = &\, | \langle{\blf},\tilde{\bu}\rangle_{V^*\times V} - \nu a_0({\bw}_0,\tilde{\bu}) - a_1({\bw}_0;{\bw}_0,\tilde{\bu})|\\
		\le &\, c_1\|{\blf}\|_{V^*}\|\tilde{\bu}\|_V + \nu\|\tilde{\bu}\|_V\|{\bw}_0\|_{H^1(\Omega;\mathbb{R}^2)} + \|\tilde{\bu}\|_V\|{\bw}_0\|_{H^1(\Omega;\mathbb{R}^2)}^2,
	\end{align*}
	where $c_1>0$ is the Poincar{\'e} constant.
	Therefore, we have the following estimate with $c = \frac{2}{\nu}\max\{ 1,c_1\}$
	\[ 
		\|\tilde{\bu}\|_{V} \le c(\|\blf\|_{V^*} + (\nu +\|{\bw}_0\|_{H^1(\Omega;\mathbb{R}^2)}^2)\|{\bw}_0\|_{H^1(\Omega;\mathbb{R}^2)}^2).
	\]
\end{proof}

We note that even though the solution seem to have been dependent on the parameter $\epsilon>0$, we point out that we only needed this dependence on the value $\epsilon=\nu/2$, this will hold true even for the upcoming result on uniqueness of solutions.We also infer from the last computations in the proof of Theorem \ref{theorem:wpnhdirichlet} that the following inqualities hold
\begin{align}
	\|\tilde{\bu}\|_{V} \le \frac{2}{\nu} \|\Phi\|_{V^*} \le c(\|\blf\|_{V^*} + (\nu +\|{\bw}_0\|_{H^1(\Omega;\mathbb{R}^2)}^2)\|{\bw}_0\|_{H^1(\Omega;\mathbb{R}^2)}^2).
	\label{estimate:dualphi}
\end{align}

\begin{theorem}
	Suppose that the assumptions in Theorem \ref{theorem:wpnhdirichlet} hold, and that the following estimate is true
	\begin{align}
		4\mathcal{B}\|\Phi\|_{V^*} < \nu^2.
		\label{estimate:uniquenessnhdirichlet}
	\end{align}
	Then the perturbed weak solution $\tilde{\bu}\in V$ of \eqref{weak:nhdirichlet} is unique.
	\label{theorem:uniquenessnhdirichlet}
\end{theorem}
\begin{proof}
	Again, we assume two solutions $\tilde{\bu}_1,\tilde{\bu}_2\in V$ of \eqref{weak:nhdirichlet}, and note that both solutions satisfy \eqref{estimate:dualphi}. This implies that the element $\tilde{\bw} = \tilde{\bu}_1-\tilde{\bu}_2\in V$ solves the equation
	\begin{align}
		\nu a_0({\bw},{\bv}) + a_1({\bw};{\bw}_0,{\bv}) + a_1({\bu}_2;{\bw},{\bv}) + a_1({\bw}_0;{\bw},{\bv}) = a_1({\bw};{\bv},{\bu}_1)\quad \forall{\bv}\in V.
		\label{weak:differencenhdirichlet}
	\end{align}
Taking ${\bv} = {\bw}$ , by virtue of Lemmas \ref{lemma:trilinearpropoerties}(ii) and \ref{lemma:midinp}, and by utilizing the estimates \eqref{estimate:trilinear} and \eqref{estimate:dualphi} yield
	\begin{align*}
		\frac{\nu}{2}\|{\bw}\|_V^2 \le&\, \nu\|{\bw}\|_V^2 + a_1({\bw};{\bw}_0,{\bw})\\
			= &\, a_1({\bw};{\bw},{\bu}_1) \le \frac{2\mathcal{B}\|\Phi\|_{V^*}}{\nu}\|{\bw}\|_V^2.
	\end{align*}
Hence the inequality $(\nu^2 - 4\mathcal{B}\|\Phi\|_{V^*})\|{\bw}\|_V^2\le 0$, and from assumption \eqref{estimate:uniquenessnhdirichlet} we infer that ${\bw} = 0$.
\end{proof}

\section{Non-stationary Problem} 
\label{sec4}
In this section, we study the evolutionary case of the Navier--Stokes equations with the convective boundary condition. As we shall see, we derive a good energy estimate which is not achievable for the usual outflow condition, i.e., the usual do-nothing boundary condition. For the sake of brevity, we shall only consider the case with a non-homogeneous Dirichlet condition, which is the dynamic version of the system considered in Section \ref{section:2.2}. In particular,  for an interval $I= (0,T)$ for $T>0$ we consider the following system
\begin{align}
	\left\{
	\begin{aligned}
		\partial_t{\bu}-\nabla\cdot\sigma({\bu},p) + ({\bu}\cdot\nabla){\bu} & = {\blf} &&\text{in }\Omega\times I,\\
		\nabla\cdot{\bu} & = 0 &&\text{in }\Omega\times I,\\
		{\bu}(0) & = {\bu}_0 &&\text{in }\Omega,\\
		{\bu} & = 0 &&\text{on }\Gamma_H\times I,\\
		{\bu} & = {\bu}_{in} &&\text{on }\Gamma_N\times I,\\
		\sigma({\bu},p){\bn} & = \frac{1}{2}({\bu}\otimes{\bu}){\bn} &&\text{on }\Gamma_1\times I.
	\end{aligned}
	\right.
	\label{system:timenhdirichlet}
\end{align}
We note that the same assumptions in Section \ref{section:2.2} for the structure of the domain still hold in this system. Furthermore, to take account the dynamicity of the system we assume that ${\blf}\in L^2(I;V^*)$, ${\bu}_{in}\in L^2(I;H^{1/2}(\Gamma_N;\mathbb{R}^2))$, and that
\begin{align}
\int_{\Gamma_N} {\bu}_{in}(t)\cdot{\bn} \du s=0\quad \text{a.e. }t\in I\text{ including }t=0.
\label{equality:timenhomogeneous0}
\end{align}
We also utilize Lemma \ref{lemma:midinp} so that for an arbitrary $\epsilon >0$ there exists ${\bw}_0:={\bw}_0(\epsilon)\in L^2(I;H^1(\Omega;\mathbb{R}^2))$ that satisfies \eqref{properties:midinp} and \eqref{estimate:midinp}. We assume for compatibility that ${\bu}_0|_{\Gamma_N} =  {\bu}_{in}(0)$. Furthermore, to apply Lemma \ref{lemma:midinp} we assume that $ {\bu}_{in}(0) \in H^{1/2}(\Gamma_N;\mathbb{R}^2) $, which in turn compels us to suppose that ${\bu}_0\in H^1(\Omega;\mathbb{R}^2)$ due to trace theorem.
\begin{remark}
Note that the usual assumption for the initial data ${\bu}_0$ is for it to be in $L^2(\Omega;\mathbb{R}^2)$ (see \cite{temam1977}), however such assumption will not be able to handle the compatibility ${\bu}_0|_{\Gamma_N} =  {\bu}_{in}(0)$.
\end{remark}

We shall call an element $\tilde{\bu}\in W^2(V)$ a perturbed weak solution of \eqref{system:timenhdirichlet} if it satisfies the following equation 
\begin{align}
	\langle \partial_t\tilde{\bu},{\bv}\rangle_{V^*\times V} + \nu a_0(\tilde{\bu},{\bv}) + a_1(\tilde{\bu};\tilde{\bu},{\bv}) + a_1({\bw}_0;\tilde{\bu},{\bv}) +  a_1(\tilde{\bu};{\bw}_0,{\bv}) = \langle\Phi,{\bv}\rangle_{V^*\times V}\quad \forall{\bv}\in V,
	\label{weak:timenhdirichlet}
\end{align}
for almost every $t\in I$, and $\tilde{\bu}(0) = {\bu}_0 - {\bw}_0(0)$ in $H$. We note that the evaluation at $t=0$ is well-defined since $W^2(V)\subset C(\overline{I};H)$, and that the element $\Phi\in L^2(I;V^*)$ is defined as in Section \ref{section:2.2} but takes into account the time dependence of its components.

\begin{theorem}
	Suppose that $\Omega\subset\mathbb{R}^2$ is of class $\mathcal{C}^1$, ${\blf}\in L^2(I;V^*)$, ${\bu}_0\in H^1(\Omega;\mathbb{R}^2)$ with $\nabla\cdot{\bu}_0=0$ in $\Omega$, and ${\bu}_{in }\in L^2(I;H^{1/2}(\Gamma_N;\mathbb{R}^2))$ with $ {\bu}_{in}(0) = {\bu}_0|_{\Gamma_N} \in H^{1/2}(\Gamma_N;\mathbb{R}^2) $ and satisfies \eqref{equality:timenhomogeneous0}. Then, there exists $\tilde{\bu}\in W^2(V)$ that solves \eqref{weak:timenhdirichlet} and satisfies the following energy estimate
	\begin{align}
		\|\tilde{\bu}\|_{W^2(V)} \le  c(\|{\blf}\|_{L^2(I;V^*)},\|{\bu}_0\|_{H^1(\Omega;\mathbb{R}^2)},\|{\bw}_0\|_{L^2(I;H^1(\Omega;\mathbb{R}^2))}).
	\end{align}
\end{theorem} 

\begin{proof}
We approach the proof by utilizing an orthonormal basis $\{{\bv}_k\}$ of $V$, and project the problem on the subspace $V_n := \mathrm{span}\{{\bv}_k\}_{k=1}^n$. Furthermore, the solution of the projected problem may be written as $\tilde{\bu}_n = \sum_{k=1}^n \alpha_k^{(n)}(t){\bv}_k $ and the projected problem can be written as the initial value problem
\begin{align*}
\left\{
\begin{aligned}
	\langle \partial_t\tilde{\bu}_n,{\bv}_k\rangle_{V^*\times V} + \nu a_0(\tilde{\bu}_n,{\bv}_k) + a_1(\tilde{\bu}_n;\tilde{\bu}_n,{\bv}_k) + a_1({\bw}_0;\tilde{\bu}_n,{\bv}_k) +  a_1(\tilde{\bu}_n;{\bw}_0,{\bv}_k) &= \langle\Phi,{\bv}_k\rangle_{V^*\times V},\\
	\tilde{\bu}_n(0) &= \tilde{\bu}_{0n},
\end{aligned}
\right.
	\label{weak:projectedtimenhdirichlet}
\end{align*}
for all $k=1,\ldots,n,$ and where $\tilde{\bu}_{0n}$ is a projection of ${\bu}_0 - {\bw}_0(0)$ over the space $V_n$ (for example by the Leray projection operator). This differential equation can be easily shown to have a solution in an interval $[0,t_n]$ by the virtue, for example, of Picard's theorem. Fortunately, the following {\it a priori} estimates will show that $t_n = T$. Since $\tilde{\bu}_n \in V_n$, we infer from the first equation of the previous system that 
\begin{align}
	\frac{1}{2}\frac{d}{dt}\|\tilde{\bu}_n\|_H^2 + \nu \|\tilde{\bu}_n\|_V^2 + a_1(\tilde{\bu}_n;{\bw}_0,\tilde{\bu}_n) \le \frac{1}{\nu}\|\Phi\|_{V^*}^2 + \frac{\nu}{4}\|\tilde{\bu}_n\|_V^2.
\end{align}
Now, by choosing the lifting of the input function ${\bu}_{in}$ with $\epsilon = \nu/4$, then by taking the integral over an interval $[0,t]$ for $t<t_n$ of both sides of the resulting inequality will give us 
\begin{align}
	 \|\tilde{\bu}_n(t)\|_H^2 + \nu\int_0^t\|\tilde{\bu}_n(s)\|_V^2 \du s \le \|\tilde{\bu}_{0n}\|_H^2 + \frac{2}{\nu}\int_0^t\|\Phi(s)\|_{V^*}^2\du s.
	 \label{equation:4.5}
\end{align}
Furthermore, from H{\"o}lder's inequality, and \eqref{estimate:trilinear} of Lemma \ref{lemma:trilinearpropoerties}, we infer that 
\begin{align*}
	\|\partial_t{\bu}_n\|_{L^2(I;V^*)} \le c((1+\|{\bu}_n\|_{L^2(I;V)})\|{\bu}_n\|_{L^2(I;V)} + \|\Phi\|_{L^2(I;V^*)}).
\end{align*}
Since $\|\tilde{\bu}_{0n}\|_H \le c \|{\bu}_0 - {\bw}_0(0)\|_H$, and from the assumptions on the external force ${\blf}$ and the lifting of the input function ${\bu}_{in}$, we infer that $\|\tilde{\bu}_n\|_{W^2_I(V)} \le c(\|{\bu}_0 - {\bw}_0(0)\|_H,\|{\Phi}\|_{V^*})$, with $I = [0,T]$. These imply that the initial value problem admits a solution over the whole interval $I$. Furthermore, estimate \eqref{equation:4.5} aids us to infer that there exists $\tilde{\bu}\in L^2(I;V)\cap L^\infty(I;H)$ such that 
\begin{align*}
	\left\{
		\begin{aligned}
			\tilde{\bu}_n & \rightharpoonup \tilde{\bu} &&\text{in }L^2(I;V),\\
			\tilde{\bu}_n & \ws \tilde{\bu} &&\text{in }L^\infty(I;H),\\
			\tilde{\bu}_n &\to \tilde{\bu}	&&\text{in }L^2(I;H),
		\end{aligned}
	\right.
\end{align*}
where the third convergence is due to the compact embedding $W^2(V)\hookrightarrow L^2(I;H)$. Using classical sequential arguments, one can prove that $\tilde{\bu}\in W^2(V)$ satisfies \eqref{weak:timenhdirichlet}. We further note that \eqref{equation:4.5} and the fact that $\|\tilde{\bu}_{0n}\|_H \le c \|{\bu}_0 - {\bw}_0(0)\|_H$ imply that the following energy estimate holds
\begin{align}
	 \|\tilde{\bu}(t)\|_H^2 + \nu\int_0^t\|\tilde{\bu}(s)\|_V^2 \du s \le \|{\bu}_0 \|_{L^2(\Omega;\mathbb{R}^2)}^2 + \|{\bw}_0(0) \|_{L^2(\Omega;\mathbb{R}^2)}^2  + \frac{2}{\nu}\int_0^t\|\Phi(s)\|_{V^*}^2\du s.
\end{align}
\end{proof}

For the uniqueness of the solution, we note that given $\tilde{\bu}_1,\tilde{\bu}_2\in W^2(V)$ that solve \eqref{weak:timenhdirichlet}, then the difference $\bw = \tilde{\bu}_1-\tilde{\bu}_2\in W^2(V)$ solves the equation
\begin{align}
			\langle \partial_t{\bw}, {\bv}\rangle_{V^*\times V} + \nu a_0({\bw},{\bv}) + a_1({\bw};\tilde{\bu}_1,{\bv}) + a_1(\tilde{\bu}_2;{\bw},{\bv}) + a_1({\bw}_0;{\bw},{\bv}) + a_1({\bw};{\bw}_0,{\bv}) & = 0,	
			\label{weak:differencetime}
\end{align}
for almost every $t\in I$ and ${\bw}(0) = 0$. By performing diagonal testing and utilizing Lemma \ref{lemma:trilinearpropoerties}(ii) on \eqref{weak:differencetime} yields
\begin{align*}
	\frac{1}{2}\frac{d}{dt}\|{\bw}(t)\|_H^2 + \nu \|{\bw}(t)\|_V^2 =  -a_1({\bw}(t); \tilde{\bu}_1(t)+{\bw}_0(t),{\bw}(t)).
\end{align*}
Integrating the equation above over the interval $[0,t]$ for $t<T$, utilizing the estimate in Lemma \ref{lemma:trilinearpropoerties} and Young's inequality, we then get 
\begin{align*}
	\|{\bw}(t)\|_H^2 \le \frac{c}{2\nu}\int_0^t \|\tilde{\bu}_1(t)+{\bw}_0(t)\|_V^2\|{\bw}(t)\|_H^2\du t.
\end{align*}
Lastly, from Gronwall's inequality we infer that ${\bw}(t)\equiv 0$ in $H$, and thus $\tilde{\bu}_1 = \tilde{\bu}_2.$

\section{Numerical Examples}\label{sec5}
In this section, we shall illustrate some numerical solutions generated by the Navier--Stokes equations with the convective boundary condition. We shall illustrate simulations of the three systems whose analyses we just previously exposed.  

For a domain $\Omega\subset\mathbb{R}^2$, we consider a regular triangulation $\mathcal{T}_h = \{K\}$ of $\overline{\Omega}$. We also consider the spaces 
\[X_h := \{{\bv}_h \in C(\overline{\Omega};\mathbb{R}^2): {\bv}_h|_K \in \mathbb{P}^2(K;\mathbb{R}^2),\, \forall K\in \mathcal{T}_h  \},\] 
\[M_h := \{q_h \in C(\overline{\Omega};\mathbb{R}): q_h|_K \in \mathbb{P}^1(K;\mathbb{R}),\, \forall K\in \mathcal{T}_h  \},\]
$V_h :=  X_h \cap H_{\Gamma_0}^1(\Omega;\mathbb{R}^d) $, $Q_h := M_h \cap L^2(\Omega;\mathbb{R})$, and $Y_h := V_h\times Q_h$, where $\mathbb{P}^m(K;\mathbb{R}^d)$ is the spaces of $k$-degree polynomial functions on $K\in \mathcal{T}_h$.

\subsection{Stationary Case: Homogeneous Dirichlet}
In this part, we shall illustrate simulations of the system considered in Section \ref{section:2.1}. 
We shall employ Newton's method to circumvent the issue of the nonlinearity of the system.
We begin with the velocity-pressure operator $\mathbb{E}_H:Y_h\to Y_h^*$ of \eqref{system:hdirichlet} given by 
\begin{align}
	\begin{aligned}
	\langle \mathbb{E}_H({\bu}_h,p_h), ({\bv}_h,q_h) \rangle_{Y_h\times Y_h^* } =&\,  \nu a_0({\bu}_h,{\bv}_h) + a_1({\bu}_h;{\bu}_h,{\bv}_h)+ b({\bv}_h,p_h)\\ & + b({\bu}_h,q_h)   -  \langle{\blf},{\bv}_h\rangle_{V_h\times V_h^*},
	\end{aligned}
	\label{vpform:hdirichlet}
\end{align}
where $b({\bu},q) := -\int_\Omega q\dive{\bu}\du x$. Since the operator $b$ satisfies the inf-sup condition, 
the first component of the solution to $\mathbb{E}_H({\bu}_h,p_h) = 0$ also solves \eqref{weak:hdirichlet} in its discretized version.  Furthermore, the Fr{\'e}chet derivative $\mathbb{E}_H':Y_h\times Y_h\to Y_h^*$ at a point $({\bu}_h,p_h)\in Y_h$ can be determined easily and is given by 
\begin{align}
	\begin{aligned}
	\langle \mathbb{E}_H'({\bu}_h,p_h)(\delta{\bu}_h,\delta p_h), ({\bv}_h,q_h) \rangle_{Y_h^*\times Y_h} =&\,  \nu a_0(\delta{\bu}_h,{\bv}_h) + a_1(\delta{\bu}_h;{\bu}_h,{\bv}_h)+ a_1({\bu}_h;\delta{\bu}_h,{\bv}_h)\\ & + b({\bv}_h,\delta p_h) + b(\delta{\bu}_h,q_h).
	\end{aligned}
	\label{frechet:hdirichlet}
\end{align}
Furthermore, $\mathbb{E}_H'({\bu}_h,p_h) $ is an isomorphism from $Y_h$ to $Y_h^*$ given that the uniqueness assumption \eqref{estimate:uniquenesshdirichlet} holds. With this regard, for a given $({\bu}_h^n,p_h^n)\in Y_h$ a Newton iterate $({\bu}_h^{n+1},p_h^{n+1})\in Y_h$ is determined as the solution to the variational equation
\begin{align}
	\begin{aligned}
	\langle \mathbb{E}_H'({\bu}_h^n,p_h^n)(\delta{\bu}_h^{n+1},\delta p_h^{n+1}), ({\bv}_h,q_h) \rangle_{Y_h\times Y_h^* } = - \langle \mathbb{E}_H({\bu}_h^n,p_h^n), ({\bv}_h,q_h) \rangle_{Y_h\times Y_h^* }\quad \forall({\bv}_h,q_h)\in Y_h,
	\end{aligned}
	\label{newton:hdirichlet}
\end{align}
where $(\delta{\bu}_h^{n+1},\delta p_h^{n+1}) = ({\bu}_h^{n+1},p_h^{n+1})-({\bu}_h^{n},p_h^{n}) $.

This system is solved in the domain $\Omega = [0,1]^2$, and the outflow boundary is chosen to be the left wall of the domain. Furthermore, we used the same external force that was used in \cite{braack2014} to illustrate the difference between the boundary flow induced by the usual do-nothing condition and our proposed convective condition, i.e., we take ${\blf} = (f_1, f_2)$ with $f_1 = \sin(x) + \sin(y)$ and $f_2 = 0$. We note that in such set-up, the system is steered in such a way that inflow occurs on the upper part of the boundary while an outflow is imposed on the lower part. Figures \ref{figure:homogeneous} and \ref{figure:homogeneous2} compares the do-nothing and convective conditions for different values of the viscosity constant $\nu$, in particular, Figure \ref{figure:homogeneous} exhibits comparisons with $\nu = 1, 1/10, 1/20, 1/30, 1/40$ while the other figure shows simulations where $\nu = 1/50, 1/60, 1/70, 1/80, 1/90$.

\begin{figure}[ht!]
 \centering
  \includegraphics[width=\textwidth]{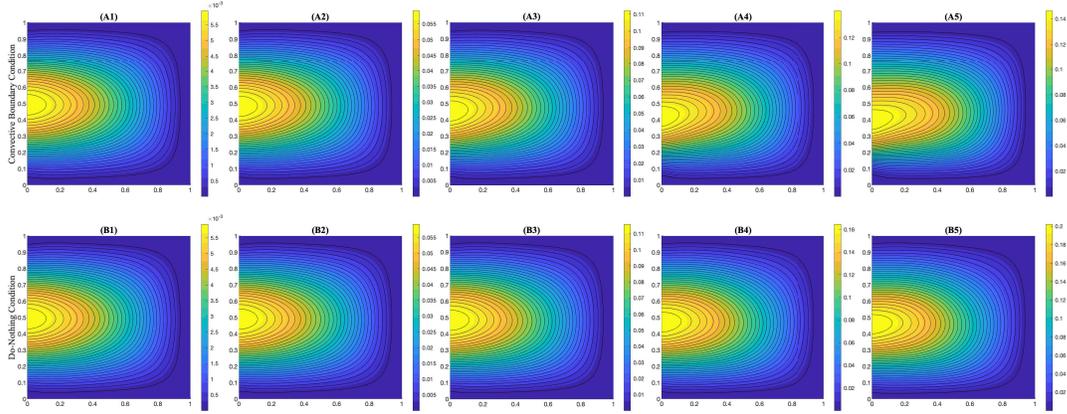}
 \caption{The figure shows simulations of system \eqref{system:hdirichlet} using CBC with viscosity constants $\nu = 1$, $\nu = 1/10$, $\nu = 1/20$, $\nu = 1/30$,$\nu = 1/40$ (A1-A5); and using the usual do-nothing condition instead of CBC on $\Gamma_1$ with the same values of $\nu$ (B1-B5)}
 \label{figure:homogeneous}
 \end{figure}

Initially, we can observe on the first three columns of Figure \ref{figure:homogeneous} (i.e., (A1)-(A3) and (B1)-(B3)) that the flows for both boundary conditions exhibit the same behaviour. This is an expected behaviour since advective effects on fluids are neglected for low values of Reynold's number, which compels both flows to mimic Stokes flow. 

Meanwhile, the right-side columns of Figure \ref{figure:homogeneous} (i.e., (A4)-(A5) and (B4)-(B5)) illustrates simulations with lower viscosity constants. In this case, the difference between the flows induced by the usual do-nothing condition and CBC are apparent. In fact, one can observe convergence of the outflow around $y = 0.1$ for the CBC, while linear outflow behaviour is observed on the system governed with the do-nothing condition.

\begin{figure}[ht!]
 \centering
  \includegraphics[width= \textwidth]{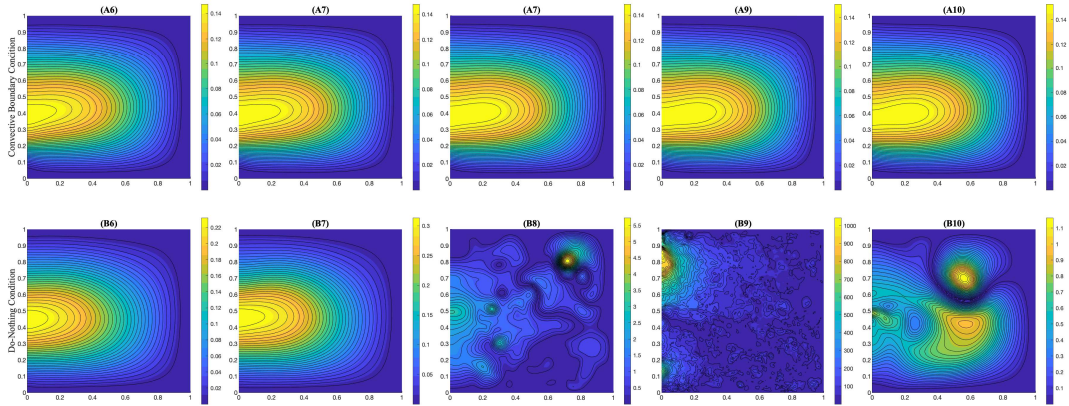}
 \caption{The figure shows simulations of system \eqref{system:hdirichlet} using CBC with viscosity constants  $\nu = 1/50$, $\nu = 1/60$, $\nu = 1/70$, $\nu = 1/80$, $\nu = 1/90$ (A6-A10); and using the usual do-nothing condition instead of CBC on $\Gamma_1$ with the same values of $\nu$ (B6-B10)}
 \label{figure:homogeneous2}
 \end{figure}

The convergence of outflow may also be observed on the figures on the first row of Figure \ref{figure:homogeneous2}, i.e., (A6)-(A10). Furthermore, if we look at (A6), (A7), (B6) and (B7) this outflow convergence causes the fluid rotation to dissipate, while the rotation induced by the do-nothing condition seems to be symmetrical. Aside from the observed difference in flow, Figure \ref{figure:homogeneous2} also shows the stability of Newton's method using CBC, i.e., Newton's iterations do not converge in the case of do-nothing condition for $\nu = 1/70, 1/80, 1/90$.

\begin{figure}[ht!]
 \centering
  \includegraphics[width=.8\textwidth]{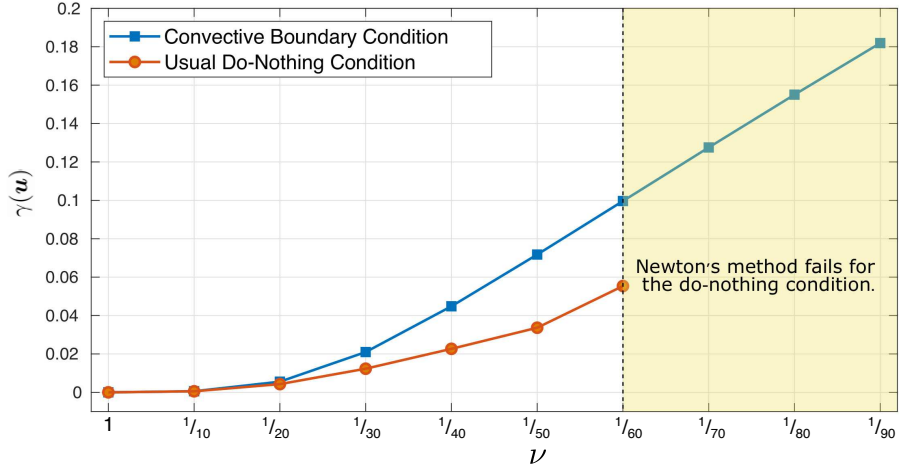}
 \caption{ Plots of the nonlinear outflow induced by CBC and the usual do-nothing boundary condition versus $x$, where $x = 0$ for $\nu = 1$, and $x = 1/(10\nu)$ otherwise.  }
 \label{figure:homogeneousoutflowcomparison}
 \end{figure}

To quantify the difference between the two boundary conditions, we also computed the nonlinear outflow $\gamma({\bu})$ of a velocity ${\bu}$ which is solved as 
\begin{align}
	\gamma({\bu}) = \int_{\Gamma_0} ({\bu}\cdot{\bn})_+{\bu} \du s,
\end{align}
where $({\bu}\cdot{\bn})_+ = {\bu}\cdot{\bn} - ({\bu}\cdot{\bn})_-$. From Figure \ref{figure:homogeneousoutflowcomparison}, we can see the similarity of the outflows for lower Reynold's number, specifically for $\nu = 1, 1/10, 1/20$. The glaring difference occurs for lower values of $\nu$, which can be attributed to the bulk motion of fluid which we have visually observed. To be precise, the high values of the nonlinear outflow is due to the fact that CBC takes into account the convective forces around the boundary.

Lastly, we point out that due to stability estimate \eqref{energy:hdirichlet}, we are assured with converging iterations for Newton's method. Using the usual do-nothing condition, on the other hand, does not give us the same assurance as the proposed boundary condition. Again, this is observed on the lower rows of Figure \ref{figure:homogeneous2} as the last three simulations fail to converge to a solution.

\begin{remark}
We end this subsection by recalling certain results in \cite{braack2014}. We point out that in the aforementioned reference, they showed lower nonlinear outflow using their directional do-nothing condition as compared to the usual do-nothing. Furthermore, they highlighted that this may have been caused by some stability property caused by the boundary condition they proposed. Nevertheless, we also mention that due to the discontinuous nature of their boundary condition, utilization of full Newton's method is quite challenging. This in fact, may cause slower convergence as compared to the boundary condition we are proposing, where we can linearize the whole nonlinearity of the expression. This fact will be the focus of future studies.
\end{remark}

\subsection{Stationary Case: Non-homogeneous Dirichlet}
In this part, we shall show numerical examples of the system considered in Section \ref{section:2.2}. Again, due to the nonlinear nature of the system, we employ Newton's method. We consider the following non-homogeneous velocity-pressure operator $\mathbb{E}_N : Y_h\to Y_h^*$ given by
\begin{align}
\begin{aligned}
	\langle \mathbb{E}_N(\tilde{\bu}_h,p_h),({\bv}_h,q_h)\rangle_{Y_h^*\times Y_h} = &\, \nu a_0(\tilde{\bu}_h,{\bv}_h) + a_1(\tilde{\bu}_h;\tilde{\bu}_h,{\bv}_h) + a_1({\bw}_{0h};\tilde{\bu}_h,{\bv}_h)\\ & + a_1(\tilde{\bu}_h;{\bw}_{0h},{\bv}_h) + b({\bv}_h, p_h) + b(\tilde{\bu}_h,q_h) - \langle \Phi,{\bv}_h\rangle_{V^*\times V},
	\end{aligned}
\label{vpform:nhdirichlet}
\end{align}
where ${\bw}_{0h}\in X_h$ is the projection of the lifting ${\bw}_0 \in H^1(\Omega;\mathbb{R}^2)$ of the input function. Similarly with the homogeneous case, we utilize the Fr{\'e}chet derivative of $\mathbb{E}_N$ to induce the Newton formulation, this gives us
 \begin{align}
	\begin{aligned}
	\langle \mathbb{E}_N'(\tilde{\bu}_h^n,p_h^n)(\delta\tilde{\bu}_h^{n+1},\delta p_h^{n+1}), ({\bv}_h,q_h) \rangle_{Y_h\times Y_h^* } = - \langle \mathbb{E}_N(\tilde{\bu}_h^n,p_h^n), ({\bv}_h,q_h) \rangle_{Y_h\times Y_h^* }\quad \forall({\bv}_h,q_h)\in Y_h,
	\end{aligned}
	\label{newton:nhdirichlet}
\end{align}
where $(\delta\tilde{\bu}_h^{n+1},\delta p_h^{n+1}) = (\tilde{\bu}_h^{n+1},p_h^{n+1})-(\tilde{\bu}_h^{n},p_h^{n}) $. Here, the derivative is explicitly solved as 
\begin{align}
\begin{aligned}
	\langle \mathbb{E}_N'(\tilde{\bu}_h,p_h)&(\delta\tilde{\bu}_h,\delta p_h),({\bv}_h,q_h)\rangle_{Y_h^*\times Y_h} =  \nu a_0(\delta\tilde{\bu}_h,{\bv}_h) + a_1(\delta\tilde{\bu}_h;\tilde{\bu}_h,{\bv}_h)+ a_1(\tilde{\bu}_h;\delta\tilde{\bu}_h,{\bv}_h) \\ &  + a_1({\bw}_{0h};\delta\tilde{\bu}_h,{\bv}_h) + a_1(\delta\tilde{\bu}_h;{\bw}_{0h},{\bv}_h) + b({\bv}_h,\delta p_h) + b(\delta\tilde{\bu}_h,q_h) - \langle \Phi,{\bv}_h\rangle_{V^*\times V}.
	\end{aligned}
\label{frechet:nhdirichlet}
\end{align}

We solve this system in a bifurcation geometry with varying values of $\nu>0$ as shown in Figure \ref{figure:nonhomogeneous}. Furthermore, the input boundary is defined as $\Gamma_N = \{(x,y)\in\mathbb{R}^2: x=0, -1/2 \le y \le 1/2\}$ with the input function ${\bu}_{in} = ((1/2 - y)(1/2+y),0)$. The output boundary $\Gamma_1$ is defined as $\Gamma_1 = \{ (x(t),y(t))\in \mathbb{R}^2: x(t) = t, y(t) = \pm (t-7.5), 6\le t\le 6.5 \}$, while the segments of the wall boundary $\Gamma_H$ are defined so as to illustrate a bifurcation geometry (see Figure \ref{figure:nonhomogeneousdomain}).

\begin{figure}[ht!]
 \centering
  \includegraphics[width=.6\textwidth]{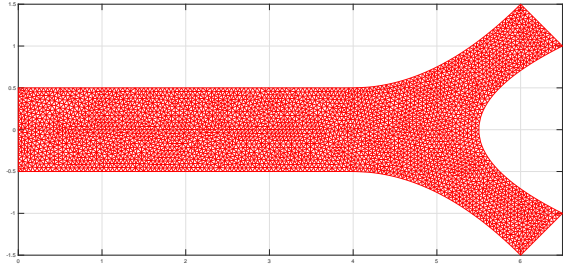}\vspace{-.2in}
 \caption{ Domain discretization for the resolution of system \eqref{system:nhdirichlet} }
 \label{figure:nonhomogeneousdomain}
 \end{figure}


We compare the flow - in terms of the streamlines - induced by \eqref{newton:nhdirichlet} with the flows induced using the do-nothing and the directional do-nothing condition. The flow from the do-nothing condition is solved by removing the boundary integrals on the definitions of $\mathbb{E}_N$ and $\mathbb{E}'_N$. A quasi-Newton method is utilized for the directional do-nothing by fully linearizing the domain quadratic term, while only considering the term $\int_{\Gamma_1} (\tilde{\bu}_h^{n}\cdot{\bn})_{-}\delta\tilde{\bu}_h^{n+1}\du s$ instead of the fully linearized version which is
\begin{align*}
\int_{\Gamma_1} (\tilde{\bu}_h^{n}\cdot{\bn})_{-}\delta\tilde{\bu}_h^{n+1} +  (\delta\tilde{\bu}_h^{n+1}\cdot{\bn})_{-} \tilde{\bu}_h^{n}\du s,
\end{align*}
since it would be challenging to solve for $\delta\tilde{\bu}_h^{n+1}$ in the term $(\delta\tilde{\bu}_h^{n+1}\cdot{\bn})_{-}$ from usual Galerkin methods.

\begin{figure}[ht!]
 \centering
  \includegraphics[width=1\textwidth]{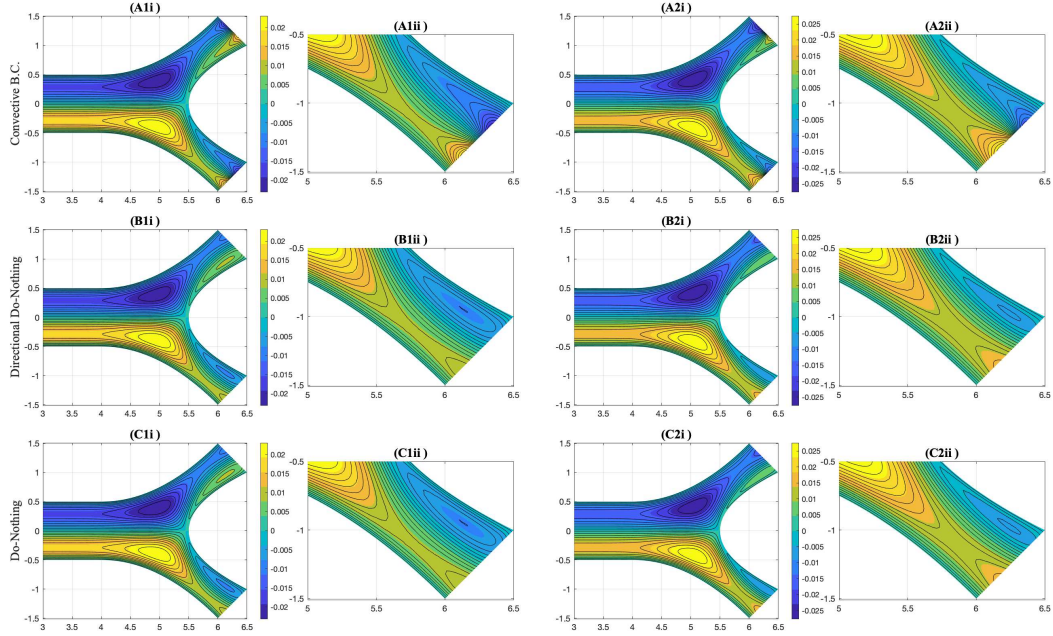}\vspace{-.5in}
 \caption{The figure shows simulations of system \eqref{system:nhdirichlet} using CBC with viscosity constants  $\nu = 1/250$ {\rm (A1i)}, $\nu = 1/1000$ {\rm (A2i)} and their respective zoomed in lower branches {\rm(A1ii)} and {\rm(A2ii)}; using the directional do-nothing condition instead of CBC on $\Gamma_1$ with the same values of $\nu = 1/250, 1/1000$ {\rm (B1i)}-{\rm(B2i)}  and their respective zoomed in lower branch {\rm(B1ii)} and {\rm(B2ii)}; and using the usual do-nothing condition instead of CBC on $\Gamma_1$ with the same values of $\nu = 1/250, 1/1000$ {\rm(C1i)}-{\rm(C2i)}  and their respective zoomed in lower branch {\rm(C1ii)} and {\rm(C2ii)} }
 \label{figure:nonhomogeneous}
 \end{figure}

As can be observed in Figure \ref{figure:nonhomogeneous} - where the values  $\nu = 1/250$ and $\nu = 1/1000$ are used - the advect effect on the boundary are now obvious as the viscosity constant are small. Although the flows using CBC and the do-nothing condition looks almost similar, we can still see bulking of fluid motion for our proposed condition on the boundary $\Gamma_1$ (see Figure \ref{figure:nonhomogeneous} (A1ii) and (A2ii)).  This fluid bulking causes the streamline rotations near the said boundary - which are present in the flows induced from the usual do-nothing and the directional do-nothing - to dissipate, whereas linear outflows are observed on the other streamlines (see Firgure \ref{figure:nonhomogeneous}(B1ii), (B2ii), (C1ii), and (C2ii)). We also mention that such bulking of fluid motion may also be observed on the Navier--Stokes flows with the total pressure boundary condition, see \cite{heywood1992} for illustration.
 
{\bf Note:} The streamline circulations observed near the point $(\pm 0.25, 5)$ are not necessarilty caused by vortices, they are rather caused by high acceleration (caused by the nonlinear term on the governing state of the velocity field ${\bu}$) which causes the streamlines seem to disconnect at lower values of contour levels. Such streamline circulations are also observed near the outflow boundary and should not be mistaken as fluid backflow.

\subsection{Time-Dependent Case}

Simulating time-dependent  problems can be done in several methods. Nevertheless, we shall employ what is known as the Lagrange--Galerkin method, which takes advantage of the approximation of the material derivative of the fluid velocity. Here,  for a given velocity field ${\bu}$, we consider a function $X:(0,T)\to \mathbb{R}^2$ that solves the characteristic equation $\frac{dX}{dt} = {\bu}(X,t). $
This implies that 
\[ 
 \partial_t{\bu} + ({\bu}\cdot\nabla){\bu} =: \frac{D{\bu}}{Dt}(X(t),t) = \frac{d}{dt}{\bu}(X(t),t),
\]
for sufficiently smooth ${\bu}$. 

Let $\Delta t$ be a time increment, and $t^n := n\Delta t$ for $n\in\mathbb{N}\backslash\{0\}$. We denote by $h^n$ the evaluation $h(\cdot,t^n)$ of a function $h:\Omega\times(0,T)$. For a point $x\in \mathbb{R}^2$, we denote by $X(\cdot;x,t^n)$ the solution to the characteristic equation with initial condition $X(t^n) = x$.  We shall utilize the upwind point of $x$ with respect to the velocity ${\bu}$ given by $\vec{X}({\bu},\Delta t)(x):= x - {\bu}\Delta t$. In fact, the point $\vec{X}({\bu}^{n-1},\Delta t)(x)$ approximates $X(t^{n-1};x,t^n)$, for $n\in\mathbb{N}$, so that 
\[
\frac{D{\bu}}{Dt} (x,t^n) = \frac{d}{dt}{\bu}(X(t),t)\big|_{t=t^n} \approx \frac{{\bu}^n - {\bu}^{n-1}\circ\vec{X}({\bu}^{n-1},\Delta t) }{\Delta t}(x).
\]
The nonlinearity on the boundary on the other will be solved using an implicit-explicit scheme, i.e.,we shall use the following approximation:
\[
	\int_{t^{n-1}}^{t^n} \int_{\Gamma_1} ({\bu}\cdot{\bn})({\bu}\cdot{\bv}) \du s\du t \approx \Delta t\int_{\Gamma_1} ({\bu}^{n-1}\cdot{\bn})({\bu}^{n}\cdot{\bv}) \du s \quad \forall{\bv}\in V.
\]

Let $N = \lfloor T/\Delta t \rfloor$ be the number of time steps, and $({\bu}_h^0,p_h^0)\in V_h\times Q_h$ be a projection of $({\bu}_0,0)\in V\times L^2(\Omega;\mathbb{R}) $, we propose a Lagrange-Galerkin scheme that approximates the solution to \eqref{system:timenhdirichlet} by solving for $\{ ({\bu}_h^n,p_h^n) \}_{n=1}^N \subset V_h\times Q_h$, that satisfies for each $n = 1,2,\ldots,N$ the equation
\begin{align}
\begin{aligned}
	\int_\Omega \frac{{\bu}^n_h - {\bu}^{n-1}_h\circ\vec{X}({\bu}^{n-1}_h,\Delta t) }{\Delta t} \cdot {\bv}_h\du x + \nu a_0({\bu}_h^n,{\bv}_h) + b({\bv}_h,p_h^n)&\\
	+ b({\bu}_h^n,q_h) - \int_{\Gamma_1} ({\bu}^{n-1}_h\cdot{\bn})({\bu}^{n}_h\cdot{\bv}_h) \du s = \int_\Omega {\blf}^n\cdot {\bv}_h \du x&,
\end{aligned}\qquad \forall({\bv}_h,q_h)\in V_h\times Q_h.
\end{align}

The analysis of existence, stability, and convergence of the approximation is beyond the scope of this paper. Nevertheless, we present an implementation of such scheme.  In particular, we illustrate flow past a cylinder. Here, the channel is a rectangle with vertices $(-1.5,-1)$, $(-1.5,1)$, $(1,6)$, and $(-1,6)$. Meanwhile the cylinder is a circle centered at $(0,0)$ and with radius $r = 0.15$. The Dirichlet data on $\Gamma_N:=\{(x,y)\in\mathbb{R}^2: x = -1.5, -1\le y\le 1 \}$ is defined as ${\bu}_{in}(t) = ((1-y)(1+y),0)$, while the outflow boundary is defined as $\Gamma_1:= \{(x,y)\in\mathbb{R}^2: x = 6, -1\le y\le 1 \} $, and the viscosity constant is $\nu = 1/250$.


 \begin{figure}[ht!]
 \centering
 \includegraphics[width=\textwidth]{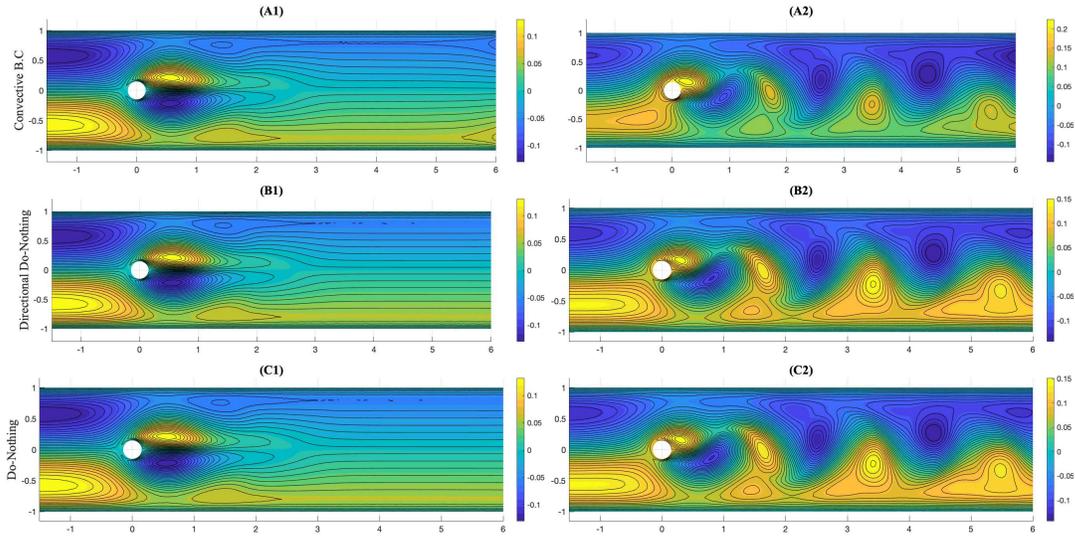} \vspace{-.5in}
 \caption{The figure shows simulations of system \eqref{system:timenhdirichlet} using CBC with viscosity constant  $\nu = 1/250$ captured at times $t=4$(A1), and $t=24$ (A2); using the directional do-nothing condition instead of CBC on $\Gamma_1$ with the same value of $\nu$ and same time-captures (B1-B2); and using the usual do-nothing condition instead of CBC on $\Gamma_1$ with the same value of $\nu$ and same time-captures (C1-C2) }
 \label{figure:timedependent}
 \end{figure}

We see from Figure \ref{figure:timedependent}(A1) that the convective effect on the boundary $\Gamma_1$ is apparent for system \eqref{system:timenhdirichlet} at time $t=4$, while the flows induced using the directional and the usual do-nothing conditions behave in a more linear way, see Figure \ref{figure:timedependent}(B1) and (C1), respectively. These phenomena are also observed during the shedding of Karman vortex at $t=24$. In particular, we observe from Figure \ref{figure:timedependent}(A2) that the vortex located around the point $(5.5,-0.5)$ seems to dissipate due to the convective boundary condition, while dominant vortex cores are observable for the directional and the usual do-nothing conditions,  see Figures \ref{figure:timedependent}(B2) and (C2), respectively.

\section{Conclusion}

In this paper, we first established existence of weak solution for the stationary systems \eqref{system:hdirichlet} and \eqref{system:nhdirichlet}. The analysis for the system with the homogeneous Dirichlet data was done in a straightforward manner, while the analysis for \eqref{weak:nhdirichlet} was accomplished with the aid of Lemma \ref{lemma:midinp} which was an analogous result to \cite[Lemma IV.2.3]{girault1986} but takes into account the estimate on the boundary integral. 
We also obtained uniqueness of solutions for both systems given that either the given source functions are sufficiently small or the viscosity constant $\nu$ is sufficiently large.

We also analyzed existence and uniqueness for the time-dependent problem. Unlike the usual assumption that the initial data ${\bu}_0$ is in $L^2(\Omega;\mathbb{R}^2)$, we found out that the necessary assumption is for it to be at least first differentiable. The reason for this is to take into account the compatibility with the Dirichlet data on $\Gamma_N$ at the time $t=0$. 

Lastly, we illustrated numerical examples for the three systems we analyzed. For the stationary system with homogeneous Dirichlet condition, we illustrated how our system induces velocity fields that take into account the convective forces on the boundary $\Gamma_1$ by comparing it with the solutions using the usual do-nothing condition. Such convective effects were also observed for the illustrations of systems \eqref{system:nhdirichlet} and \eqref{system:timenhdirichlet}, which we both compared to flows induced by the directional and usual do-nothing conditions. The time-dependent problems were all simulated by utilizing a Lagrange-Galerkin method.

\section*{Appendix}
\appendix
\renewcommand{\thesection}{A}

The resolution of the stream function ${\psi} : \Omega\to\mathbb{R}$ on boundaries with Dirichlet data can be quite easily done by solving the values of said function explicitly. However, solving for such function with Convective (also called Robin) and Neumann type boundary conditions can be {\it challenging} since there is no explicit form for the velocity field on such boundaries. Here, we lay out the method by which we solved the streamlines shown in the simulations above. In particular, we solve the following variational problem: {\it For a given velocity field ${\bu}\in L^2(\Omega;\mathbb{R}^2)$, find $\psi\in H^1(\Omega;\mathbb{R})$ that sarisfies}
\begin{align}
	\int_{\Omega}\nabla\psi\cdot \nabla\phi \du x = \int_{\Omega}{\bu}\cdot(\nabla\times \phi)\quad \forall \phi\in H^1_{\Gamma_0}(\Omega;\mathbb{R}),
	\label{weak:streams}
\end{align}
{\it with $\psi|_{\Gamma_0} = g$ for an appropriate $g\in H^{1/2}(\Gamma_0;\mathbb{R})$, which is chosen according to the Dirichlet data imposed on the velocity field.} Here the curl for a scalar valued function is defined as $\nabla\times\phi = (\frac{\partial \phi}{\partial x_2},-\frac{\partial \phi}{\partial x_1})$.  The impetus for such formulation is the assumption that ${\bu} = \nabla\times \psi$, this implies that the derivative of $\psi$ on the boundary $\Gamma_1$ in the direction ${\bn}$ can then be solved as $\nabla\phi \cdot {\bn} = -{\bn}\times {\bu}$, hence we have the following system
\begin{align}
	\left\{
		\begin{aligned}
			-\Delta \phi & = \nabla\times{\bu} &&\text{in }\Omega,\\
					\phi & = g&&\text{on }\Gamma_0,\\
					\frac{\partial\phi}{\partial {\bn}} & = -{\bn}\times{\bu} &&\text{on }\Gamma_1.			
		\end{aligned}
	\right.
	\label{system:streams}
\end{align}

Indeed, by multiplying a test function $\phi \in \mathcal{W}(\Omega;\mathbb{R}) $ to the first equation in \eqref{system:streams} and utilizing Green's first and curl \cite[Theorem 3.29]{monk2003} identities, we get
\begin{align*}
	\int_{\Omega}\nabla\psi\cdot\nabla\phi \du x - \int_{\Gamma_0}\frac{\partial\psi}{\partial{\bn}}\phi\du s = \int_{\Omega} {\bu}\cdot(\nabla\times \phi) \du x + \int_{\Gamma_0} ({\bn}\times{\bu})\phi \du s.
\end{align*}
By using the boundary condition on $\Gamma_1$ from \eqref{system:streams}, the boundary integrals in the equation above cancels out, therefore gives us \eqref{weak:streams}.

\bibliographystyle{siamplain}
\bibliography{report}

\end{document}